\documentclass[12pt]{amsart}
\usepackage{amsmath, amscd}
\usepackage{amsfonts}

\begin{document}
\numberwithin{equation}{section}

\def\1#1{\overline{#1}}
\def\2#1{\widetilde{#1}}
\def\3#1{\widehat{#1}}
\def\4#1{\mathbb{#1}}
\def\5#1{\frak{#1}}
\def\6#1{{\mathcal{#1}}}

\newcommand{\Lie}[1]{\ensuremath{\mathfrak{#1}}}
\newcommand{\LieL}{\Lie{l}}
\newcommand{\LieH}{\Lie{h}}
\newcommand{\LieG}{\Lie{g}}
\newcommand{\de}{\partial}
\newcommand{\R}{\mathbb R}
\newcommand{\FH}{{\sf Fix}(H_p)}
\newcommand{\al}{\alpha}
\newcommand{\tr}{\widetilde{\rho}}
\newcommand{\tz}{\widetilde{\zeta}}
\newcommand{\tk}{\widetilde{C}}
\newcommand{\tv}{\widetilde{\varphi}}
\newcommand{\hv}{\hat{\varphi}}
\newcommand{\tu}{\tilde{u}}
\newcommand{\tF}{\tilde{F}}
\newcommand{\debar}{\overline{\de}}
\newcommand{\Z}{\mathbb Z}
\newcommand{\C}{\mathbb C}
\newcommand{\Po}{\mathbb P}
\newcommand{\zbar}{\overline{z}}
\newcommand{\G}{\mathcal{G}}
\newcommand{\So}{\mathcal{S}}
\newcommand{\Ko}{\mathcal{K}}
\newcommand{\U}{\mathcal{U}}
\newcommand{\B}{\mathbb B}
\newcommand{\oB}{\overline{\mathbb B}}
\newcommand{\Cur}{\mathcal D}
\newcommand{\Dis}{\mathcal Dis}
\newcommand{\Levi}{\mathcal L}
\newcommand{\SP}{\mathcal SP}
\newcommand{\A}{\mathcal O^{k+\alpha}(\overline{\mathbb D},\C^n)}
\newcommand{\CA}{\mathcal C^{k+\alpha}(\de{\mathbb D},\C^n)}
\newcommand{\Ma}{\mathcal M}
\newcommand{\Ac}{\mathcal O^{k+\alpha}(\overline{\mathbb D},\C^{n}\times\C^{n-1})}
\newcommand{\Acc}{\mathcal O^{k-1+\alpha}(\overline{\mathbb D},\C)}
\newcommand{\Acr}{\mathcal O^{k+\alpha}(\overline{\mathbb D},\R^{n})}
\newcommand{\Co}{\mathcal C}
\newcommand{\Hol}{{\sf Hol}(\mathbb H, \mathbb C)}
\newcommand{\Aut}{{\sf Aut}(\mathbb D)}
\newcommand{\D}{\mathbb D}
\newcommand{\id}{\sf id}
\newcommand{\oD}{\overline{\mathbb D}}
\newcommand{\oX}{\overline{X}}
\newcommand{\loc}{L^1_{\rm{loc}}}
\newcommand{\la}{\langle}
\newcommand{\ra}{\rangle}
\newcommand{\thh}{\tilde{h}}
\newcommand{\N}{\mathbb N}
\newcommand{\kd}{\kappa_D}
\newcommand{\Hr}{\mathbb H}
\newcommand{\ps}{{\sf Psh}}
\newcommand{\Hess}{{\sf Hess}}
\newcommand{\subh}{{\sf subh}}
\newcommand{\harm}{{\sf harm}}
\newcommand{\ph}{{\sf Ph}}
\newcommand{\tl}{\tilde{\lambda}}
\newcommand{\gdot}{\stackrel{\cdot}{g}}
\newcommand{\gddot}{\stackrel{\cdot\cdot}{g}}
\newcommand{\fdot}{\stackrel{\cdot}{f}}
\newcommand{\fddot}{\stackrel{\cdot\cdot}{f}}
\def\v{\varphi}
\def\Re{{\sf Re}\,}
\def\Im{{\sf Im}\,}

\newtheorem{theorem}{Theorem}[section]
\newtheorem{lemma}[theorem]{Lemma}
\newtheorem{proposition}[theorem]{Proposition}
\newtheorem{corollary}[theorem]{Corollary}

\theoremstyle{definition}
\newtheorem{definition}[theorem]{Definition}
\newtheorem{example}[theorem]{Example}

\theoremstyle{remark}
\newtheorem{remark}[theorem]{Remark}
\numberwithin{equation}{section}

\title[Boundary behavior]{Boundary behavior of infinitesimal generators in the unit ball}
\author[F. Bracci]{Filippo Bracci$^\dag$}
\address{F. Bracci: Dipartimento Di Matematica\\
Universit\`{a} di Roma \textquotedblleft Tor Vergata\textquotedblright\ \\
Via Della Ricerca Scientifica 1, 00133 \\
Roma, Italy} \email{fbracci@mat.uniroma2.it}
\author[D. Shoikhet]{David Shoikhet}
\address{D. Shoikhet: Department of Mathematics\\
ORT Braude College \\ 21982 Karmiel,
Israel}
\email{davs@braude.ac.il}
\date\today
\thanks{$\dag$ Partially supported by the ERC grant ``HEVO - Holomorphic Evolution Equations'' n. 277691}

\begin{abstract}
We prove a Julia-Wolff-Carath\'eodory type theorem for
infinitesimal generators on the unit ball in $\mathbb{C}^{n}$.
Moreover, we study jets expansions at the boundary and give necessary and sufficient conditions on such jets
for an infinitesimal generator to generate a group
of automorphisms of the ball.
\end{abstract}



\maketitle

\section{Introduction}

The classical Julia-Wolff-Carath\'eodory theorem (see, {\sl e.g.} \cite{Ab-T, Co-Ma, RS, Shb}) is  the most powerful tool for studying properties of bounded holomorphic functions of the unit disc $\D$ of $\C$ at a given boundary point. This theorem has been generalized to the unit ball $\B^n$ of $\C^n$ by W. Rudin (see \cite{Ru}) and to strongly (pseudo)convex domains and other domains in $\C^n$ by other authors, notably by M. Abate  (see \cite{Abate}, \cite{Ab-T}. See also \cite{Abanew} for the most recent and complete survey on the subject).

In what follows we are mainly interested in the case of mappings fixing a boundary point. Since the group of automorphisms of $\B^n$ acts bi-transitively on $\de \B^n$, without loss of generality we restrict our attention to the point $e_1=(1,0\ldots, 0)\in \de \B^n$.

The maps we are working with are not assumed to be
continuous up to the boundary, thus we have to specify the meaning of
the term ``boundary fixed point''. In higher dimensions, in fact,
different approaches to boundary limits are possible. We recall
them here briefly (see \cite{Abate}, \cite{Ru}  for more
information).

Let $R\geq 1$ and let $K(e_1,R):=\{z\in \B^n: |1-z_1|\leq
\frac{R}{2}(1-\|z\|^2)\}$ be a {\sl Kor\'anyi region of vertex $e_1$ and amplitude $R$} (see
\cite[Section 5.4.1]{Ru}, \cite{Co-Ma}). In \cite[Section 2.2.3]{Abate} a slightly different but essentially equivalent definition is given and used. In order not to excessively burden the notation, since we are only working at $e_1$, from now on, when we talk about Kor\'anyi regions, we will always mean Kor\'anyi regions of vertex $e_1$.

Let $f: \B^n \to \C^n$
be a holomorphic map. We say that $f$ has {\sl $K$-limit} $L$ at
$e_1$ -- and we write $K\hbox{-}\lim_{z\to e_1}f(z)=L$ -- if for
each sequence $\{z_k\}\subset \B^n$ converging to $e_1$ such that
$\{z_k\}$  belongs eventually to some Kor\'anyi region, it follows
that $f(z_k)\to L$. We say that $f$ has {\sl restricted $K$-limit}
$L$ at $e_1$ -- and we write $\angle_K\lim_{z\to e_1}f(z)=L$ -- if
for each sequence $\{z_k\}\subset \B^n$ converging to $e_1$ such
that $\|z_k-\la z_k,e_1\ra e_1\|^2/(1-|\la z_k,e_1\ra|^2)\to 0$
and $\la z_k, e_1\ra\to 1$ non-tangentially in $\D$ it follows
that $f(z_k)\to L$. Finally, we say that $f$ has {\sl
non-tangential limit} $L$ at $e_1$ and we write $\angle\lim_{z\to
e_1}f(z)=L$, if for each sequence $\{z_k\}\subset \B^n$ converging
non-tangentially to $e_1$ -- {\sl i.e.}, such that there exists
$C>0$ with $\|z_k-e_1\|\leq C (1-\|z_k\|^2)$ for all $k\geq 1$ --
it follows that $f(z_k)\to L$.

One can show that
\[
K\hbox{-}\lim_{z\to e_1}f(z)=L\Longrightarrow \angle_K\lim_{z\to
e_1}f(z)=L\Longrightarrow\angle\lim_{z\to e_1}f(z)=L,
\]
but the converse to any of these implications  is not true in general.

A holomorphic self-map $f:\B^n\to \B^n$ has a {\sl boundary regular fixed point} at $e_1$ if $\angle \lim_{z\to e_1}f(z)=e_1$ and \[
\al_f(e_1):=\liminf_{z\to e_1}\frac{1-\|f(z)\|}{1-\|z\|}<+\infty.
\]

Now we can formulate the Julia-Wolff-Carath\'eodory Theorem for
$\B^n$ for  boundary regular fixed points in the way we need in
this paper. As is customary, we denote by $\{e_1,\ldots, e_n\}$
the standard orthonormal basis in $\C^n$ (the symbol $e_1$ denotes
thus both the point and the direction).

\begin{theorem}[Rudin]\label{RudinJWC}
Let $f:\B^n\to \B^n$ be holomorphic. Suppose that $e_1$ is a boundary regular fixed point for $f$. Then
$K\hbox{-}\lim_{z\to e_1} f(z)=e_1$. Moreover,
\begin{itemize}
\item[(1$^{'}$)] $\langle df_z(e_1), e_1\rangle$ and $\langle df_z(e_h), e_k\rangle$ are bounded in any Kor\'anyi region for $h,k=2,\ldots, n$.
\item[(1$^{''}$)] $\langle df_z(e_j), e_1\rangle/(1-z_1)^{1/2}$ is bounded  in any Kor\'anyi region for $j=2,\ldots, n$.
\item[(1$^{'''}$)] $(1-z_1)^{1/2}\langle df_z(e_1), e_j\rangle$  is bounded  in any Kor\'anyi region for $j=2,\ldots, n$.
\item[(2)] $\angle_K\lim_{z\to e_1}\frac{ 1-\la f(z),e_1\ra}{1-z_1}=\al_f(e_1)$,
\item[(3)] $\angle_K\lim_{z\to e_1}\langle df_z(e_1), e_1\rangle=\al_f(e_1)$,
\item[(4)] $\angle_K\lim_{z\to e_1}\langle df_z(e_j), e_1\rangle=0$ for $j=2,\ldots, n$.
\item[(5)] $\angle_K\lim_{z\to e_1}\frac{\la f(z), e_j\ra}{(1-z_1)^{1/2}}=0$ for $j=2,\ldots, n$.
\item[(6)] $\angle_K\lim_{z\to e_1}(1-z_1)^{1/2}\langle df_z(e_1), e_j\rangle=0$ for $j=2,\ldots, n$.
\end{itemize}
\end{theorem}

One can interpret Julia-Wolff-Carath\'eodory's theorem as a description of the first jet of a holomorphic self-map of the unit ball at a boundary regular fixed point.

One of the aims of the present paper is to give a corresponding
theorem for  infinitesimal generator (that is, {\sl $\R$-semicomplete holomorphic  vector fields}) on
$\B^n$ having a ``regular singularity'' at $e_1$.

A holomorphic vector field $G:\B^n \to \C^n$ is said to be an {\sl infinitesimal generator} if the Cauchy problem
\begin{equation}\label{Cauchy}
\begin{cases}
\stackrel{\bullet}{x}(t)=G(x(t))\\
x(0)=z_0
\end{cases}
\end{equation}
has a solution $x_{z_0}:[0,+\infty)\ni t\mapsto x(t)$ for all $z_0\in \B^n$. If this is the case, the map $\phi:[0,+\infty)\times \B^n\mapsto \B^n$ given by $\phi_t(z):=x_z(t)$ is real analytic and $z\mapsto \phi_t(z)$ is a univalent holomorphic self-map of $\B^n$ for all fixed $t\in[0,+\infty)$. The family $(\phi_t)$ is a (continuous) semigroup, namely a continuous morphism of semigroups between $(\R^+,+)$ endowed with the Euclidean topology and $({\sf Hol}(\B^n,\B^n),\circ)$ endowed with the topology of uniform convergence on compacta.

Conversely, any (continuous) semigroup of holomorphic self-maps of
$\B^n$ is associated uniquely to an infinitesimal generator.
Interior fixed points of the semigroups correspond to
singularities of the vector field. At the boundary, the situation is
more complicated (see Sections \ref{due} and \ref{tre}). For the
time being, we say that $e_1$ is a {\sl boundary regular null
point} (or BRNP for short) if it is a boundary regular fixed point
for the associated semigroup of holomorphic self-maps and we say
that $\beta\in \R$ is the {\sl dilation} of $G$ at $e_1$ if the
flow $\phi_1$ of $G$ at the time $1$ has boundary dilation
coefficient $\al_{\phi_1}(e_1)=e^{\beta}$ (see Definition
\ref{defBRNP} for a definition of BRNP which does not involve the
associated semigroup).

Now, a version of the Julia-Wolff-Carath\'eodory Theorem for
infinitesimal generators which we are going to prove is the
following:

\begin{theorem}\label{JWC}
Let $G:\B^n\to \C^n$ be an infinitesimal generator. Suppose that
\begin{equation}\label{ipo}
\begin{split}
(\ast)\quad&\B^n\ni z\mapsto\frac{|\la G(z),e_1\ra|}{|z_1-1|} \quad \hbox{is  bounded in any Kor\'anyi region and} \\ (\ast\ast)\quad& \B^n\ni z\mapsto\frac{|\langle G(z),e_j\rangle|}{|z_1-1|^{1/2}} \quad \hbox{is  bounded in any Kor\'anyi region for $j=2,\ldots, n$.}
\end{split}
\end{equation}
Then $e_1$ is a boundary regular null point for $G$. Moreover, let $\beta\in\R$ denote the dilation of $G$ at $e_1$. Then
\begin{itemize}
\item[(1$^{'}$)] $\la dG_z(e_1),e_1\ra$ and $\la dG_z(e_h),e_k\ra$ are bounded in any Kor\'anyi region for $h,k=2,\ldots, n$,
\item[(1$^{''}$)] $\la dG_z(e_j), e_1\ra /(1-z_1)^{1/2}$ is bounded in any Kor\'anyi region for $j=2,\ldots, n$,
\item[(1$^{'''}$)] $(1-z_1)^{1/2}\la dG_z(e_1), e_j\ra$ is bounded in any Kor\'anyi region for $j=2,\ldots, n$,
\item[(2)] $\angle_K\lim_{z\to e_1}\frac{ \la G(z),e_1\ra}{z_1-1}=\beta$,
\item[(3)] $\angle_K\lim_{z\to e_1}\langle dG_z(e_1), e_1\rangle=\beta$,
\item[(4)] $\angle_K\lim_{z\to e_1}\langle dG_z(e_j), e_1\rangle=0$ for $j=2,\ldots, n$.
\end{itemize}
\end{theorem}

In the case where the infinitesimal generator extends smoothly
past $e_1$, Theorem \ref{JWC} is a consequence of Theorem
\ref{RudinJWC} applied to the associated semigroup. However, if no
regularity is assumed, this way of proceeding does not seem to be
possible. Our proof, in fact, does not involve the associated
semigroup, but it is based on the properties of infinitesimal
generators, and it is contained in Section \ref{sectJWC}. In
particular, we shall prove an intermediate version  of the
Julia-Wolff-Carath\'eodory theorem assuming only
hypothesis \eqref{ipo}.$(\ast)$ (see Proposition \ref{JWC-1}). In Example \ref{esempio}, we give an example of an infinitesimal generator which satisfies \eqref{ipo}.$(\ast)$ but not \eqref{ipo}.$(\ast\ast)$ and for which some  implications of Theorem \ref{JWC} do not hold. In Subsection \ref{discuto} we discuss the (dis)similarities between Theorem \ref{RudinJWC} and Theorem \ref{JWC} and some natural open questions raised up from this work.

Next, in Section \ref{sec-higher}, assuming a $C^3$ regularity at the BRNP $e_1$,  we describe the jets space of infinitesimal generators, giving complementary results to the ones obtained in \cite{BZ} for local biholomorphisms of strongly (pseudo)convex domains. In particular, we are interested in finding (minimal, pointwise) necessary and sufficient conditions  for an infinitesimal generator to generate a group of automorphisms of $\B^n$. In case of an interior singularity, the condition is rather simple: an infinitesimal generator $G$ with a singularity at $z_0\in \B^n$ generates a group of automorphisms of $\B^n$ if and only if the spectrum of $dG_{z_0}$ is contained in the imaginary axis $i\R$.

In the case where the singularity is at the boundary, we prove the
following result:

\begin{theorem}\label{rigidity}
Let $G$ be an infinitesimal generator on $\B^n$ of class $C^3$ at $e_1$. Assume that $e_1$ is a boundary regular null point with dilation $\beta\in \R$. Then $G$ generates a group of automorphisms if and only if the following conditions are satisfied:
\begin{enumerate}
  \item $\Re \la \frac{\de G}{\de z_k}(e_1),e_k\ra =\frac{\beta}{2}$, for $k=2,\ldots, n$,
  \item $\Re \la \frac{\de^2 G}{\de z_1\de z_k}(e_1),e_1\ra =\beta$ for $k=1,\ldots, n$,
  \item $\la \frac{\de^2 G}{\de z_1\de z_k}(e_1),e_h\ra=0$ for $2\leq k<h\leq n$,
  \item $\Re \la \frac{\de^3 G}{\de z_1^3}(e_1),e_1\ra=0$.
\end{enumerate}
Moreover, if the previous conditions are satisfied, then $G\equiv 0$ if and only if $\beta=0$, $\la \frac{\de G}{\de z_k}(e_1),e_k\ra =0$ for $k=2,\ldots, n$ and $\la \frac{\de^2 G}{\de z_1^2}(e_1),e_1\ra =0$.
\end{theorem}

The assumption on the $C^3$ regularity of $G$ at $e_1$ can be
lowered by assuming the existence of an expansion of $G$ in any
Kor\'anyi region with vertex $e_1$, but for the sake of clarity, we
will deal only with the $C^3$ case.

The previous result belongs to the family of so-called ``rigidity phenomena'', where some minimal conditions on the maps/infinitesimal generators of $\B^n$ at one point imply certain specific forms. For instance, the well known Burns-Krantz rigidity theorem \cite{BK} states that a holomorphic self-map of the unit ball which is the identity up to the third order at a boundary point, is the identity {\sl tout court}. Such a result has been extended later to infinitesimal generators (see \cite{ELRS}, and also \cite{Ca}), in the following way:  an infinitesimal generator in $\B^n$ which is $0$ up to the third order at a boundary point of $\B^n$ is identically zero. In a sense, Theorem \ref{rigidity} is a quantitative version of such rigidity phenomena.

The main idea for the proof is to transfer the information on $G$
to a family of infinitesimal generators on $\D$ by means of a
method which we call  ``slice reduction'' (see Section \ref{tre}),
first introduced in \cite{BCD} and implemented here.

Finally, in Section \ref{quadratic} we show with a couple of examples that, contrarily as one might expect, the slice reductions do not preserve the boundary expansion: while in the one dimensional case the quadratic expansion at a BRNP of an infinitesimal generator is always an infinitesimal generator which generates a semigroup of linear fractional maps, in higher dimension this is no longer the case. Moreover, even if the quadratic expansion is an infinitesimal generator, the generated semigroup might not be linear fractional.

\medskip

This work was carried out while both authors where visiting the
Mittag-Leffler Institute during the program ``Complex Analysis
and Integrable Systems'' in Fall 2011. Both authors thank the
organizers and the Institute for the kind hospitality and the
atmosphere experienced there.

\smallskip

The authors  thank Mark Elin, Marina Levenshtein, and Jasmin Raissy for useful comments on a preliminary version of the manuscript.

They also warmly thank the referee for his/her precious comments and remarks which improved a lot the paper. In particular they are in debts for his/her suggestion to use weights in the statement of Theorem \ref{JWC} and for having found a mistake in the original statement of Theorem \ref{rigidity}.

\section{Infinitesimal generators on the unit ball and BRNP's}\label{due}

Infinitesimal generators have been characterized in several ways. In the unit disc $\D$, the following powerful characterization is due to Berkson-Porta formula \cite{BP}: a holomorphic vector field $g:\D\to \C$ is an infinitesimal generator if and only if there exist $\tau\in \oD$ and $p:\D\to \{z\in \C: \Re z\geq 0\}$ such that
\begin{equation}\label{Berkson-Porta}
g(\zeta)=(\tau-\zeta)(1-\overline{\tau}\zeta)p(\zeta).
\end{equation}

In the multi-dimensional case, several equivalent
characterizations are given both using Euclidean inequalities (see
\cite{RS}), the Kobayashi metric (see \cite{Abate}) and
pluripotential theory (see \cite{BCD}). In what follows, we need a
characterization only for boundary regular fixed points, which we
are going to define.

The function
\[
u_{\B^n}(z):=-\frac{1-\|z\|^2}{|1-z_1|^2}
\]
is the pluricomplex Poisson kernel with a pole at $e_1$ and its
sublevel sets $\{u_{\B^n}(z)<-1/R\}$ for $R>0$ are called {\sl
horospheres} with center $e_1$ and radius $R$.

Recall that a function $f:\B^n\to \C^m$ is $C^k$ at $e_1$ if $f$ and all its partial derivatives up to order $k$ extend continuously to $e_1$. As the horospheres are smooth ellipsoids, by Whitney's extension theorem, this is equivalent to saying that for each horosphere $E$ with center $e_1$ there exists a function $\tilde{f}$ (depending on $E$) of order $C^k$ defined in a open neighborhood of $\overline{E}$ such that $\tilde{f}|_{E}\equiv f$.

The following characterization of infinitesimal generators in terms of the function $u_{\B^n}$ has been proved in \cite[Theorem 3.11]{BCD}:

\begin{theorem}\label{char}
Let $G:\B^n\to \C^n$ be holomorphic and $C^1$ at $e_1$. If $d(u_{\B^n})_z\cdot G(z)\leq 0$ for all $z\in \B^n$ then $G$ is an infinitesimal generator.
\end{theorem}

Now we define BRNPs:

\begin{definition}\label{defBRNP}
Let $G$ be an infinitesimal generator in $\B^n$. The point $e_1$ is  a {\sl boundary regular null point}, or BRNP for short, if there exists $b\in \R$  such that
\[
(du_{\B^n})_z\cdot G(z)+ b u_{\B^n}(z)\leq 0\quad \forall z\in \B^n.
\]
The number
\[
\beta:=-\inf_{z\in \B^n} \frac{(du_{\B^n})_z\cdot G(z)}{u_{\B^n}(z)}
\]
is called  the {\sl dilation} of $G$ at $e_1$.
\end{definition}

According to \cite[Theorem 0.4]{BCD} (see also \cite{ERS}), if $G$
is an infinitesimal generator with the associated semigroup
$(\phi_t)$ then $e_1$ is a BRNP for $G$ with dilation $\beta$ if
and only if for all $t\geq 0$ it follows
\[
u_{\B^n}(\phi_t(z))\leq e^{-t\beta}u_{\B^n}(z)\quad \forall z\in \B^n.
\]
The number $e^{t\beta}$ is the so-called {\sl boundary dilation coefficient} of $\phi_t$ at $e_1$. The previous inequality means that a horosphere of center $e_1$ and radius $R>0$ is mapped into a horosphere with center $e_1$ and radius $e^{t\beta}R$.

\section{Slice reduction of infinitesimal generators and BRNPs}\label{tre}

Let
\[
\mathcal L_{e_1}:=\{v\in\C^n : \|v\|=1, \langle v, e_1\rangle =\al>0\}.
\]
Also, let
\begin{equation}\label{geodesic}
\v_v(\zeta):=\al(\zeta-1)v+e_1.
\end{equation}
It is easy to see that $\v_v:\D \to \B^n$ is holomorphic, and it is a complex geodesic, in the sense that it is an isometry
between the Poincar\'e distance in $\D$ and the Kobayashi distance in $\B^n$. Furthermore, it is well known (see, {\sl e.g.} \cite{Abate} and \cite[Section 1]{BPT}) that any complex geodesic $\eta:\D\to \B^n$ extends holomorphically through the boundary and moreover, if $e_1\in \eta(\de \D)$, then there exists an automorphism $\theta$ of the unit disc such that $\eta\circ \theta$ is of the form \eqref{geodesic}.

\begin{remark}\label{plinio}
A direct computation shows that $u_{\B^n}(\v_v(\zeta))=\frac{1}{\al^2}u_{\D}(\zeta)$ for all $\zeta\in \D$.
\end{remark}

For a vector $w\in \C^n$ we use the notation $w=(w_1,w'')\in \C\times \C^{n-1}$.
The holomorphic map $\rho_v: \B^n \to \C^n$ defined by
\begin{equation}\label{retract}
\rho_v(z_1,z''):=\left(\frac{z_1+\frac{1}{\al}\la z'', v''\ra +\frac{1-\al^2}{\al^2}(1-z_1), -\frac{(1-z_1)}{\al} v''}{1+\frac{1}{\al}\la z'', v''\ra +\frac{1-\al^2}{\al^2}(1-z_1)}\right)
\end{equation}
has the property that $\rho_v(\B^n)=\v_v(\D)$ and, moreover, $\rho_v \circ \v_v(\zeta)=\v_v(\zeta)$ for all $\zeta\in \D$.

Finally, we let $\tr_v:=\v_v^{-1}\circ \rho_v : \B^n\to \D$, {\sl i.e.}
\begin{equation}\label{projection}
\tr_v(z_1,z'')=1+\frac{\frac{1}{\al^2}(z_1-1)}{1+\frac{1}{\al}\la z'', v''\ra + \frac{1-\al^2}{\al^2}(1-z_1)}=\frac{\al\langle z,v\rangle}{1-z_1+\al\langle z, v\rangle}.
\end{equation}

Note that for every $w=(w_1,w'')\in \C^n$ it follows that
\[
d(\tr_v)_{\v_v(\zeta)}(w_1,w'')=\frac{1}{\al^2}w_1+\frac{1-\al^2}{\al^2}(\zeta-1)w_1-\frac{1}{\al}(\zeta-1)\la w'',v''\ra.
\]

\begin{definition}
Let $G:\B^n\to \C^n$ be a holomorphic vector field. Let
\[
g_v(\zeta):=d(\tr_v)_{\v_v(\zeta)}(G(\v_v(\zeta))).
\]
 We call the holomorphic vector field $g_v:\D\to \C$ the {\sl slice reduction of $G$ to $v$}.
\end{definition}

More explicitly

\begin{equation}\label{explicit}
\begin{split}
g_v(\zeta)&=\frac{1}{\al^2}G_1(\v_v(\zeta))+\frac{1-\al^2}{\al^2}(\zeta-1)G_1(\v_v(\zeta))-\frac{1}{\al}(\zeta-1)\la G''(\v_v(\zeta)),v''\ra\\
&=\frac{1}{\al^2}\zeta G_1(\v_v(\zeta))-\frac{\zeta-1}{\al}\la G(\v_v(\zeta)),v\ra.
\end{split}
\end{equation}

The following version of Julia's lemma for infinitesimal generators  was proved in \cite[Theorem 0.4]{BCD} (see also \cite[Theorem p.403]{ERS} and, for the one-dimensional case, see  \cite[Theorem 1]{CDP}):

\begin{theorem}\label{one-gv}
Let $G:\B^n\to \C^n$ be an infinitesimal generator. Then the following are equivalent:
\begin{enumerate}
\item $G$ has BRNP at $e_1$ and dilation $\beta_0\leq \beta\in \R$,
\item $d(u_{\B^n})_z\cdot G(z)+\beta u_{\B^n}(z)\leq 0$ for all $z\in \B^n$,
\item $\displaystyle{\frac{\Re \la G(z),z\ra}{1-\|z\|^2}- \Re \frac{\la G(z),e_1\ra }{1-z_1}}\leq\frac{\beta}{2}$,  for all $z\in \B^n$,
\item for each $v\in \mathcal L_{e_1}$ the slice reduction $g_v$ is an infinitesimal generator of the unit disc with BRNP at $1$ and dilation $\leq \beta$.
\item there exists $C>0$ such that for all $v\in \mathcal L_{e_1}$ it follows
\[
\limsup_{(0,1)\ni r\to 1}\frac{|g_v(r)|}{1-r}\leq C.
\]
Moreover, if the previous condition is satisfied, then $1$ is a BRNP for $g_v$ and the following non-tangential limit exist
\[
\angle\lim_{\zeta\to 1}g_v'(\zeta)=\angle\lim_{\zeta\to 1}\frac{g_v(\zeta)}{\zeta-1}=\beta_v\in \R,
\]
with $\beta_v\leq \beta$ and
\[
\beta_0=\sup_{v\in \mathcal L_{e_1}}\beta_v.
\]
\end{enumerate}
\end{theorem}

A sufficient condition for the existence of BRNP, which we will
use in the sequel, is contained in the following (see \cite{ERS}):

\begin{theorem}\label{semign}
Let $G:\B^n\to \C^n$ be an infinitesimal generator. Assume that
\begin{enumerate}
  \item $\lim_{(0,1)\ni r\to 1}G(re_1)=0$,
  \item $\liminf_{r\to 1}\Re \frac{\la G(re_1),e_1\ra}{r-1}<+\infty$.
\end{enumerate}
Then
\[
\lim_{(0,1)\ni r\to 1}\frac{\la G(re_1),e_1\ra}{r-1}=\beta\in \R
\]
and  $e_1$ is a BRNP for $G$ with  dilation $\beta$.
\end{theorem}

Slice reductions of holomorphic vector fields preserve
pluricomplex Green and Poisson functions of strongly convex
domains, as shown in \cite{BCD} (see also \cite{Ca} for somewhat
explicit computations). As a consequence, a holomorphic vector
field is an infinitesimal generator if and only if all its slice
reductions (with respect to all points $\tau\in\de \B^n$) are
infinitesimal generators in the unit disc. In what follows, we
need only a boundary version of this fact, which we prove here
explicitly for the unit ball. We start with the following:

\begin{lemma}\label{equal}
Let $G:\B^n\to \C^n$ be holomorphic. Then, for all $v=(\al,v'')\in \mathcal L_{e_1}$, and for all $\zeta\in \D$
\begin{enumerate}
                  \item $\displaystyle{\frac{\Re \la G(\v_v(\zeta)),\v_v(\zeta)\ra}{1-\|\v_v(\zeta)\|^2}- \Re \frac{\la G(\v_v(\zeta)),e_1\ra }{1-\la\v_v(\zeta),e_1\ra}=
\frac{\Re \la g_v(\zeta),\zeta\ra}{1-|\zeta|^2}- \Re \frac{g_v(\zeta) }{1-\zeta}}$
                  \item for all $\delta\in \R$ it follows
                  \[
                  d(u_{\D})_{\zeta}\cdot g_v(\zeta)+\delta u_\D(\zeta)=\al^2[d(u_{\B^n})_{\v_v(\zeta)}\cdot G(\v_v(\zeta))+\delta u_{\B^n}(\v_v(\zeta))].
                  \]
                \end{enumerate}
\end{lemma}

\begin{proof}
(1) We have $1-\|\v_v(\zeta)\|^2=\al^2(1-|\zeta|^2)$ and $1-\la\v_v(\zeta),e_1\ra=\al^2(1-\zeta)$. Write $G_1$ for $G_1(\v_v(\zeta))$ and $G_2$ for $\la G(\v_v(\zeta)),v\ra$. Then, taking into account that for all $a\in \C$ it holds $\Re (a\zeta)+\Re (a\overline{\zeta})=2\Re\zeta \Re a$, and expanding (1), we have
\begin{equation*}
\begin{split}
&\frac{1}{\al}\frac{\Re (G_2\overline{\zeta})}{1-|\zeta|^2}-\frac{1}{\al}\frac{\Re G_2}{1-|\zeta|^2}+\frac{1}{\al^2}\frac{\Re G_1}{1-|\zeta|^2}-\frac{1}{\al^2}\frac{\Re G_1}{|1-\zeta|^2}+\frac{1}{\al^2}\frac{\Re (G_1\overline{\zeta})}{|1-\zeta|^2}-  \frac{|\zeta|^2}{\al^2}\frac{\Re G_1}{1-|\zeta|^2}\\+&\frac{|\zeta|^2}{\al}\frac{\Re G_2}{1-|\zeta|^2}-\frac{1}{\al}\frac{\Re (G_2\overline{\zeta})}{1-|\zeta|^2}+\frac{1}{\al^2}\frac{\Re(G_1\zeta)}{|1-\zeta|^2}-\frac{|\zeta|^2}{\al^2}\frac{\Re G_1}{|1-\zeta|^2}-\frac{1}{\al}\frac{\Re(G_2\zeta)}{|1-\zeta|^2}+\frac{|\zeta|^2}{\al}\frac{\Re G_2}{|1-\zeta|^2}\\+&\frac{1}{\al}\frac{\Re G_2}{|1-\zeta|^2}-\frac{1}{\al}\frac{\Re (G_2\overline{\zeta})}{|1-\zeta|^2}
= \left( \frac{1-|\zeta|^2}{\al^2(1-|\zeta|^2)}+\frac{-1+2\Re\zeta-|\zeta|^2}{\al^2|1-\zeta|^2} \right) \Re G_1\\+&\left(\frac{|\zeta|^2-1}{\al(1-|\zeta|^2)}+\frac{|\zeta|^2-2\Re\zeta+1}{\al|1-\zeta|^2}\right)\Re G_2=0,
\end{split}
\end{equation*}
as we wanted.

(2) A direct computation shows that
\[
d(u_{\B^n})_z\cdot G(z)=-2 \Re \left(\frac{\la G(z),e_1\ra}{1-z_1}\right) \frac{1-\|z\|^2}{|1-z_1|^2}+2\Re \frac{\la G(z),z\ra}{|1-z_1|^2}.
\]
Hence, the result follows from (1)  taking into account Remark
\ref{plinio}. Also, see \cite[Eq. (4.7) p.45]{BCD}, where such a
formula has been proved for strongly convex domains.
\end{proof}

Also, we need the following lemma which will be useful to move from BRNP with dilation $>0$ to BRNP with dilation $\leq 0$:
\begin{lemma}\label{H-hyp}
Let $\beta\in \R$. Define $H_\beta:\B^n\to \C^n$ by
\begin{equation}\label{hhh}
H_\beta(z)=\frac{\beta}{2} (e_1-z_1z).
\end{equation}
Then $H_\beta$ generates a group of (hyperbolic) automorphisms of $\B^n$, with BRNP at $e_1$ with dilation $-\beta$ and
\begin{equation}\label{oroH}
d(u_{\B^n})_z\cdot H_\beta(z)-\beta u_{\B^n}(z)\equiv 0\quad \forall z\in \B^n.
\end{equation}
 Moreover, for all $v\in \mathcal L_{e_1}$ the slice reduction is
\[
 h_v(\zeta)=\frac{\beta}{2}(1-\zeta^2)=-\beta(\zeta-1)-\frac{\beta}{2}(\zeta-1)^2.
\]
\end{lemma}

\begin{proof}
It is well known that  $H_\beta$ is a generator of a group of (hyperbolic) automorphisms (see, {\sl e.g.} \cite{BCDl}) with BRNP at $e_1$ and dilation $-\beta$. Hence $-H_\beta$ is a generator of a group of automorphisms having BRNP at $e_1$ with dilation $\beta$. Applying Theorem \ref{one-gv}.(2) at both $H_\beta$ and $-H_\beta$ we get \eqref{oroH}.

The form of the slice reductions is a direct computation from the very definition.
\end{proof}

In the paper we will use several times the following trick, whose proof is immediate from Theorem \ref{one-gv} and Lemma \ref{H-hyp}, which we state here for the reader convenience:

\begin{corollary}\label{trick}
Let $G:\B^n\to \C^n$ be an infinitesimal generator and assume $e_1$ is a BRNP for $G$, with dilation $\delta$. Let $\beta\in \R$ and let $H_\beta$ be given by \eqref{hhh}. Then $G+H_\beta$ is an infinitesimal generator in $\B^n$ with $e_1$ as BRNP and dilation $\delta-\beta$.
\end{corollary}

Now we can prove a boundary characterization of infinitesimal generators at BRNP:

\begin{proposition}\label{cbDW}
Let $G:\B^n\to \C^n$ be  holomorphic and $C^1$ at $e_1$. Then the following are equivalent:
\begin{enumerate}
\item $G$ is an infinitesimal generator with BRNP at $e_1$ and dilation $\leq\beta \in \R$,
\item for each $v\in \mathcal L_{e_1}$ the slice reduction $g_v$ is an infinitesimal generator of the unit disc with BRNP at $1$ and dilation $\leq \beta$.
\item $d(u_{\B^n})_z\cdot G(z)+\beta u_{\B^n}(z)\leq 0$ for all $z\in \B^n$.
\end{enumerate}
\end{proposition}

\begin{proof}
(1) implies (2) and (3) by Theorem \ref{one-gv}.

If either (2) or (3) holds, the only aim is to show that $G$ is an
infinitesimal generator, because then (1) follows from Theorem
\ref{one-gv}.

Let $F:=G+H_\beta$, where $H_\beta$ is given by \eqref{hhh}.

Assume (2) holds.  By Theorem \ref{one-gv} it follows that $d(u_\D)_\zeta\cdot g_v(\zeta)+\beta u_\D(\zeta)\leq 0$ for all $\zeta\in\D$ and $v\in \mathcal L_{e_1}$. Hence, by Lemma \ref{equal} and \eqref{oroH}  it is easy to see that $d(u_{\B^n})_z\cdot F(z)\leq 0$ for all $z\in \B^n$. The same conclusion is obtained  directly if (3) holds. By Theorem \ref{char} it follows that $F$ is an infinitesimal generator, and so does $G=F-H_\beta$, because infinitesimal generators in the ball form a cone (see \cite[Corollary  2.5.29]{Abate}).
\end{proof}

Finally, we have the following characterization of generators of groups which we will use later.

\begin{proposition}\label{gruppo}
Let $G:\B^n\to \C^n$ be holomorphic and $C^1$ at $e_1$. Let $\beta\in \R$. The following are equivalent:
\begin{enumerate}
  \item $G$ generates a group of automorphisms of $\B^n$ with BRNP $e_1$ and dilation $\beta$,
  \item $d(u_{\B^n})_z\cdot G(z)+\beta u_{\B^n}(z)\equiv 0$,
  \item for each $v\in \mathcal L_{e_1}$ the infinitesimal generator $g_v$ generates a group of automorphisms of $\D$ with BRNP $1$ and dilation $\beta$,
  \item for each $v\in \mathcal L_{e_1}$ it holds $d(u_{\D})_\zeta\cdot g_v(\zeta)+\beta u_{\D}(\zeta)\equiv 0$.
\end{enumerate}
Moreover, $G\equiv 0$ (hence the group it generates is the trivial group of automorphisms $\phi_t(z)\equiv z$ for all $t\geq 0$) if and only if for each $v\in \mathcal L_{e_1}$ it follows $g_v\equiv 0$. If this is the case then $\beta=0$.
\end{proposition}

\begin{proof} (2) is equivalent to (4) by Lemma \ref{equal}.

If (1) holds then $-G$ is an infinitesimal generator on $\B^n$ with BRNP at $e_1$ and dilation $-\beta$. Hence Theorem \ref{one-gv} applied to $G$ and $-G$ implies (2).
If (2) holds, then (1) follows from Proposition \ref{cbDW} applied to $G$ and $-G$.
Similarly,  (3) is equivalent to (4).

Finally, by \eqref{explicit} it is easy to see that  $G\equiv 0$  if and only if $g_v\equiv 0$ for all $v\in \mathcal L_{e_1}$.
\end{proof}

\section{The Julia-Wolff-Carath\'eodory Theorem for infinitesimal generators}\label{sectJWC}

As a matter of notation, we write $\angle \lim $ for non-tangential limits, $\angle_K\lim$ for restricted $K$-limits and $K-\lim$ for $K$-limits.

\begin{proposition}\label{JWC-1}
Let $G$ be an infinitesimal generator on $\B^n$. Suppose  $\lim_{(0,1)\ni r\to 1}G(re_1)=0$ and
\begin{equation}\label{ipo1}
\B^n\ni z\mapsto\frac{|\la G(z),e_1\ra|}{|z_1-1|}\quad \hbox{is  bounded in any Kor\'anyi region.}
\end{equation}
Then $e_1$ is a BRNP for $G$. Moreover, if $\beta\in \R$ is the dilation of $G$ at $e_1$, then
\begin{itemize}
\item[(1$^{'}$)]  $\B^n\ni z\mapsto \la dG_z(e_1),e_1\ra$ is bounded in any Kor\'anyi region,
\item[(1$^{''}$)] $\B^n\ni z\mapsto \frac{\la dG_z(e_j),e_1\ra}{|z_1-1|^{1/2}}$ is bounded in any Kor\'anyi region for $j=2,\ldots, n$,
\item[(2)] $\angle_K\lim_{z\to e_1}\frac{\la G(z),e_1\ra}{z_1-1}=\beta$,
\item[(3)] $\angle_K\lim_{z\to e_1}\langle dG_z(e_1), e_1\rangle=\beta$.
\end{itemize}
\end{proposition}

\begin{proof}
By hypotheses of the theorem clearly guarantee that the hypotheses of Theorem \ref{semign} are  satisfied so that $e_1$ is a BRNP for $G$.

(1$^{'}$)  The proof is based on an
application of the Cauchy formula and it is similar to the one
given by Rudin for the case of holomorphic self-maps of the unit
ball (see \cite[p.180]{Ru}). For the sake of completeness, we
sketch it here.

Let $R\geq 1$ and let $K(e_1,R)=\{z\in \B^n: |1-z_1|\leq \frac{R}{2}(1-\|z\|^2)\}$ be a Kor\'anyi region. Let $R'>R$ and $\delta:=\frac{1}{3}(\frac{1}{R}-\frac{1}{R'})$. By \cite[Lemma 8.5.5]{Ru} if $z\in K(e_1,R)$ and $\lambda\in \C$ is such that $|\lambda|\leq \delta|1-z_1|$ and  $u''\in \C^{n-1}$ is such that $\|u''\|\leq \delta |1-z_1|^{1/2}$ then $(z_1+\lambda, z''+u'')\in K(e_1, R')$.

Now, fix $z\in K(e_1,R)$ and let $r=r(z):=\delta |1-z_1|$. By the Cauchy formula
\begin{equation*}
\begin{split}
\la dG_z(e_1), e_1\ra &=\frac{1}{2\pi i}\int_{|\zeta|=r} \frac{\la G(z_1+\zeta, z''), e_1\ra}{\zeta^2}d\zeta \\&=\frac{1}{2\pi}\int_0^{2\pi} \frac{\la G(z_1+re^{i\theta}, z''), e_1\ra}{z_1+re^{i\theta}-1}\left(1-\frac{1-z_1}{re^{i\theta}} \right)d\theta.
\end{split}
\end{equation*}
Now, by the choice of $r$, the points $(z_1+re^{i\theta},z'')\in
K(e_1, R')$, hence by \eqref{ipo1} there exists a constant $C>0$
(which depends only on $R, R'$) such that
\[
\frac{|\la
G(z_1+re^{i\theta}, z''), e_1\ra|}{|z_1+re^{i\theta}-1|}\leq C.
\]
Also, $|1-\frac{1-z_1}{re^{i\theta}}|\leq 1+1/\delta$. Hence, the
function $K(e_1,R)\ni z\mapsto \la dG_z(e_1), e_1\ra$ is bounded.

(1$^{''}$) We argue as before, but, fixed $z\in K(e_1,R)$, we take $r=r(z):=\delta |1-z_1|^{1/2}$. Hence, for $j=2,\ldots, n$, we have
\begin{equation*}
\begin{split}
\frac{\la dG_z(e_j), e_1\ra}{|1-z_1|^{1/2}} &=\frac{1}{2\pi i|1-z_1|^{1/2}}\int_{|\zeta|=r} \frac{\la G(z+\zeta e_j), e_1\ra}{\zeta^2}d\zeta \\&=\frac{1}{2\pi\delta}\int_0^{2\pi} \frac{\la G(z+re^{i\theta} e_j), e_1\ra}{|1-z_1|}e^{-i\theta} d\theta.
\end{split}
\end{equation*}
By the choice of  $r$, the points $z+re_j\in
K(e_1, R')$, $j=2,\ldots, n$, and we can conclude as before.

(2) Let us consider the slice reduction $g_{e_1}(\zeta)=\la
G(\zeta e_1), e_1\ra$. By Theorem \ref{semign} it follows that
$\lim_{(0,1)\ni r\to 1}g_{e_1}(r)/(r-1)=\beta$. Since the function $\B^n
\ni z\mapsto \la G(z), e_1\ra /(z_1-1)$ is bounded in any Kor\'anyi
region by \eqref{ipo1}, $\check{\hbox{C}}$irca's theorem
\cite[Theorem 8.4.8]{Ru} implies (2).

(3) By Theorem \ref{one-gv}, we have $\lim_{(0,1)\ni r\to 1}g'_{e_1}(r)=\beta$, that is, $\lim_{(0,1)\ni r\to 1}\la dG_{re_1}(e_1),e_1\ra =\beta$. By (1) the map $\B^n\ni z\mapsto \la dG_z(e_1),e_1\ra$ is bounded in any Kor\'anyi region, and once again (3) follows by $\check{\hbox{C}}$irca's theorem \cite[Theorem 8.4.8]{Ru}.
\end{proof}

Assuming slightly more regularity at $e_1$ we can prove the following intermediate result:

\begin{proposition}\label{beta}
Let $G$ be an infinitesimal generator on $\B^n$. Suppose $\angle\lim_{z\to e_1}G(z)=0$ and
\begin{equation}
\B^n\ni z\mapsto\frac{|\la G(z),e_1\ra|}{|z_1-1|}\quad \hbox{is  bounded in any Kor\'anyi region.}
\end{equation}
Then $e_1$ is a BRNP for $G$ and $1$ is a BRNP for $g_v$ for all $v\in \mathcal L_{e_1}$. Moreover, if $\beta\in \R$ denotes the dilation of $G$ at $e_1$ and $\beta_v$ denotes the dilation of $g_v$ at $1$, then for all $v\in \mathcal L_{e_1}$ it follows $\beta_v=\beta$.
\end{proposition}

\begin{proof}
Let $v\in \mathcal L_{e_1}$. Let $g_v$ be the slice reduction to $v$ of $G$. Write $G=(G_1,G'')$. Taking into account that  for all $v\in \mathcal L_{e_1}$ the curve $(0,1)\ni r\mapsto \v_v(r)$ tends to $e_1$ non-tangentially it follows that  $\lim_{(0,1)\ni r\to1}G(\v_v(r))\to 0$. By Theorem \ref{one-gv} and Proposition~\ref{JWC-1}
\begin{equation*}
\begin{split}
\beta_v&=\lim_{(0,1)\ni r\to 1}\frac{g_v(r)}{r-1}\\&=\lim_{(0,1)\ni r\to 1}\frac{\frac{1}{\al^2}G_1(\v_v(r))+\frac{1-\al^2}{\al^2}(r-1)G_1(\v_v(r))-\frac{1}{\al}(r-1)\la G''(\v_v(r)),v''\ra}{r-1}
\\&=\frac{1}{\al^2}\lim_{(0,1)\ni r\to 1}\frac{G_1(\v_v(r))}{r-1}
=\frac{1}{\al^2}\lim_{(0,1)\ni r\to 1}\frac{G_1(\v_v(r))}{\la \v_v(r),e_1\ra -1}\frac{\la \v_v(r),e_1 \ra-1}{r-1}\\&=\frac{1}{\al^2}\beta \al^2=\beta,
\end{split}
\end{equation*}
and we are done.
\end{proof}

\begin{proof}[Proof of Theorem \ref{JWC}] The hypothesis \eqref{ipo} implies that $\angle \lim_{z\to e_1}G(z)=0$ and \eqref{ipo1}. Thus, Proposition \ref{JWC-1} applies and (1$^{''}$), (2) and (3) follow.

(1$^{'}$) The boundness of $\la dG_z(e_1),e_1\ra$ in any Kor\'anyi region follows again from Proposition \ref{JWC-1}. The proof that $\B^n\ni z\mapsto \la dG_z(e_h),e_k\ra$ is bounded in any Kor\'anyi region for $h,k=2,\ldots, n$ is similar to the proof of (1$^{''}$) in Proposition \ref{JWC-1}. Thus, we just sketch it here. Let $R,R',\delta$ as in the proof of Proposition \ref{JWC-1}. Fix $z\in K(e_1,R)$ and let $r=r(z):=\delta |1-z_1|^{1/2}$. Then for $h,k=2,\ldots, n$,
\begin{equation*}
\begin{split}
\la dG_z(e_h), e_k\ra &=\frac{1}{2\pi i}\int_{|\zeta|=r} \frac{\la G(z+\zeta e_h), e_k\ra}{\zeta^2}d\zeta =\frac{1}{2\pi}\int_0^{2\pi} \frac{\la G(z+re^{i\theta} e_h), e_k\ra}{r}e^{-i\theta} d\theta\\&=
\frac{1}{2\pi\delta}\int_0^{2\pi} \frac{\la G(z+re^{i\theta} e_h), e_k\ra}{|1-z_1|^{1/2}}e^{-i\theta} d\theta.
\end{split}
\end{equation*}
By the choice of  $r$, the points $z+re_h\in K(e_1, R')$, $h=2,\ldots, n$. Hence \eqref{ipo}.$(\ast\ast)$ guarantees that $z\mapsto \la dG_z(e_h), e_k\ra$ is bounded in $K(e_1,R)$.

(1$^{'''}$) We retain the notations introduced in the proof of Proposition \ref{JWC-1}. Fix $z\in K(e_1,R)$ and let $r=r(z):=\delta |1-z_1|$. Then, for $j=2,\ldots, n$
\begin{equation}\label{boh}
\begin{split}
|1-z_1|^{1/2}\la dG_z(e_1), e_j\ra &=\frac{|1-z_1|^{1/2}}{2\pi i}\int_{|\zeta|=r} \frac{\la G(z_1+\zeta,z''), e_j\ra}{\zeta^2}d\zeta \\&=\frac{1}{2\pi\delta}\int_0^{2\pi} \frac{\la G(z_1+re^{i\theta},z''), e_j\ra}{|1-(z_1+re^{i\theta})|^{1/2}}e^{-i\theta} \left|\frac{1-(z_1+re^{i\theta})}{1-z_1}\right|^{1/2} d\theta.
\end{split}
\end{equation}
Again, by the choice of  $r$, the points $(z_1+re^{i\theta},z'')\in K(e_1, R')$. Since $|1-(z_1+re^{i\theta})|/|1-z_1|\leq 1+\delta$, hypothesis \eqref{ipo}.$(\ast\ast)$ guarantees that $z\mapsto |1-z_1|^{1/2}\la dG_z(e_1), e_j\ra$ is bounded in $K(e_1,R)$.

(4) Let $v\in \mathcal L_{e_1}$. Let $g_v$ be the slice reduction to $v$ of $G$. Since $\v_v'(\zeta)=\al v$,
\begin{equation*}
\begin{split}
g'_v(\zeta)&=\frac{1}{\al}\la dG_{\v_v(\zeta)}(v),e_1\ra +(\zeta-1)\frac{1-\al^2}{\al}\la dG_{\v_v(\zeta)}(v),e_1\ra
-(\zeta-1)\la dG''_{\v_v(\zeta)}(v),v''\ra\\&+\frac{1-\al^2}{\al^2}\la G(\v_v(\zeta)), e_1\ra -\frac{1}{\al}\la G''(\v_v(\zeta)),v''\ra.  \\
\end{split}
\end{equation*}
By Proposition \ref{beta} and  Theorem \ref{one-gv} we have $\lim_{(0,1)\ni r \to 1}g'_v(r)=\beta$.

Taking into account that $(0,1)\ni r\mapsto \v_v(r)$ tends to $e_1$ non-tangentially, we have  $\lim_{r\to 1} G(\v_v(r))=0$ and $r \mapsto \la dG_{\v_v(r)}(v),e_1\ra$ is bounded by (1$^{'}$) and (1$^{''}$). Moreover, by (1$^{'}$) and (1$^{'''}$)  it follows that for $v=\al e_1+\sum_{j=2}^n v_j e_j\in \mathcal L_{e_1}$
\begin{equation*}
\begin{split}
\lim_{r\to 1}(r-1)\la dG''_{\v_v(r)}(v),v''\ra&=\lim_{r\to 1}(r-1)\left(\sum_{j=2}^n \al \overline{v_j}\la dG_{\v_v(r)}(e_1),e_j\ra\right.\\&\left.+\sum_{h,k=2}^n v_h\overline{v_k}\la dG_{\v_v(r)}(e_h),e_k\ra \right)=
\lim_{r\to 1}\sum_{j=2}^n \al \overline{v_j}(r-1)\la dG_{\v_v(r)}(e_1),e_j\ra\\&=
\sum_{j=2}^n\al \overline{v_j} \lim_{r\to 1}\frac{(r-1)}{(1-\la \v_v(r),e_1\ra)^{1/2}}(1-\la \v_v(r),e_1\ra)^{1/2}\la dG_{\v_v(r)}(e_1),e_j\ra\\&
=-\sum_{j=2}^n\frac{\overline{v_j}}{\al}\lim_{r\to 1}(r-1)^{1/2}(1-\la \v_v(r),e_1\ra)^{1/2}\la dG_{\v_v(r)}(e_1),e_j\ra=0.
\end{split}
\end{equation*}
Therefore,
\begin{equation}\label{g-der}
\beta=\lim_{(0,1)\ni r\to 1}g'_v(r)=\frac{1}{\al}\lim_{(0,1)\ni r\to 1}\la dG_{\v_v(r)}(v),e_1\ra.
\end{equation}
Expanding \eqref{g-der}, and taking into account (3), we have
\begin{equation*}
\begin{split}
\beta&=\frac{1}{\al}\lim_{(0,1)\ni r\to 1}\left(\la dG_{\v_v(r)}(\al e_1),e_1\ra+\sum_{j=2}^n v_j\la dG_{\v_v(r)} (e_j), e_1\ra\right)\\&=\beta+\frac{1}{\al}\lim_{(0,1)\ni r\to 1}\left(\sum_{j=2}^n v_j\la dG_{\v_v(r)} (e_j), e_1\ra\right),
\end{split}
\end{equation*}
from which it follows that, for all choices of $v$
\[
\lim_{(0,1)\ni r\to 1}\left(\sum_{j=2}^n v_j\la dG_{\v_v(r)} (e_j), e_1\ra\right)=0.
\]
For the arbitrariness of $v$ we have
\[
\lim_{(0,1)\ni r\to 1}\la dG_{\v_v(r)} (e_j), e_1\ra=0 \quad j=2,\ldots, n.
\]
Since the function $\B^n\ni z\mapsto \la dG_{z}(e_j),e_1\ra$ is bounded in every Kor\'anyi region and  has limit $0$ along a non-tangential curve, by $\check{\hbox{C}}$irca's theorem \cite[Theorem 8.4.8]{Ru}, it has restricted $K$-limit $0$, and this proves (4).
\end{proof}

\begin{example}\label{esempio}(cfr. \cite[Example 4.2]{BCD}).
Let $G(z_1,z_2):=(0, -z_2/(1-z_1))$. Then $G$ is an infinitesimal generator in $\B^n$ with BRNP $e_1$ and dilation $\beta=0$. Note that $|\la G(z), e_1\ra |\equiv 0$, and Proposition \ref{JWC-1} applies.
However, $\angle\lim_{z\to e_1}G(z)$ does not exist. In fact, $dG_z$ is not bounded in any Kor\'anyi region: a direct computation shows that
\[
dG_z=\left(
       \begin{array}{cc}
         0 & 0 \\
         -\frac{z_2}{(1-z_1)^2} & -\frac{1}{1-z_1} \\
       \end{array}
     \right).
\]
Moreover, given $v=(\al, v_2)\in \mathcal L_{e_1}$ it is easy to see that
\[
g_v(\zeta)=(1-\frac{1}{\al^2})(\zeta-1),
\]
hence, the dilation $\beta_v$ of $g_v$  at $1$ is $1-1/\al^2$. Thus, $\beta_v<0=\beta$ for all $v\in \mathcal L_{e_1}\setminus \{e_1\}$, and $\beta_{e_1}=0$.
\end{example}

\subsection{(Dis)similarities between the Julia-Wolff-Carath\'eodory theorems for maps and for infinitesimal generators and open questions}\label{discuto}

Hypothesis \eqref{ipo} is stronger than the corresponding starting hypothesis in Rudin's theorem, which involves only finiteness of the liminf defining $\al_f(e_1)$. In fact, part of the work in proving Rudin's theorem is devoted to show that such a condition, via Julia's lemma, implies boundness of suitable functions in any Kor\'anyi region. Julia's lemmas for infinitesimal generators (see Theorem \ref{one-gv}) are however -- and, in a certain sense, very naturally -- weaker than those for self-mappings and this forced us to use such a stronger hypothesis. We do not know whether there exists any weaker condition in terms of liminf of some function of $G$ which assures (and it is equivalent to) hypothesis \eqref{ipo}.

It would be also interesting to find an example (if any) of an infinitesimal generator satisfying  the hypothesis of Proposition \ref{beta} but not hypothesis \eqref{ipo}.

Moreover, with our techniques, we are not able to prove (or disprove) for infinitesimal generators the statements  corresponding to (5) and (6) of Theorem \ref{RudinJWC}. Namely, under the hypothesis of Theorem \ref{JWC}, we do not know whether  for $j=2,\ldots, n$ it holds
\begin{equation}\label{nose}
\angle_K\lim_{z\to e_1}\frac{\la G(z), e_j\ra}{(1-z_1)^{1/2}}=0, \quad \angle_K\lim_{z\to e_1}(1-z_1)^{1/2}\langle dG_z(e_1), e_j\rangle=0.
\end{equation}
By $\check{\hbox{C}}$irca's theorem \cite[Theorem 8.4.8]{Ru} and  Theorem \ref{JWC}.(1$^{'''}$) and using \eqref{boh}, these results hold if one can prove that for $j=2,\ldots, n$
\begin{equation}\label{radile}
\lim_{(0,1)\ni r\to 1}\frac{\la G(re_1), e_j\ra}{(1-r)^{1/2}}=0
\end{equation}
In the case of Rudin's theorem, the corresponding radial limit is proven using Julia's lemma and the strong constrain of sending the ball into itself. 

Let $f:\B^n\to \B^n$ be a holomorphic self-map having a BRFP at $e_1$. Then $G(z):=f(z)-z$ is an infinitesimal generator (see \cite[Corollary 3.3.1]{Shb} and \cite{RS}) and, using Theorem \ref{RudinJWC},  it is not hard to see that $G$ satisfies \eqref{ipo} at $e_1$. Again by Theorem \ref{RudinJWC} it is easy to see that $G$ satisfies \eqref{radile}, and hence \eqref{nose}. Therefore, for the dense subclass of infinitesimal generators of the form $f(z)-z$ with $f:\B^n\to \B^n$ holomorphic, the full analogue of Rudin's theorem holds.

Thus, it is reasonable to believe that even in the general case  \eqref{radile} holds and should follow from Julia's lemma for infinitesimal generators and the condition of being an infinitesimal generators. However, we are not able to prove the result.

\section{Higher order jets of  generators at BRNPs}\label{sec-higher}

Let $G$ be an infinitesimal generator on $\B^n$ having a boundary regular null point (BRNP) at $e_1$ and assume $G$ is $C^3$ at $e_1$. We can expand $G$ in the form
\begin{equation}\label{expansion}
G(z)=T(z-e_1)+Q_2(z-e_1)+Q_3(z-e_1)+o(|z-e_1|^3),
\end{equation}
where $Q_j$ is a $n$-tuple of homogeneous polynomial of degree $j$ for $j=2,3$.
Then by Theorem \ref{JWC} we can write,
\begin{equation}\label{T}
T=\left(
    \begin{array}{cccc}
      \beta & 0 &\ldots & 0\\
      t_2 & s_{22} &\ldots & s_{2n} \\
      \vdots & \vdots & \vdots & \vdots\\
      t_n & s_{n2} &\ldots & s_{nn}
    \end{array}
  \right)
\end{equation}
where $\beta\in \R$ is the dilation of $G$ at $e_1$ and $t_j, s_{jk}\in \C$. We set $S=(s_{jk})_{j,k=2,\ldots, n}$.

Also, we write $(x,y)\in \C\times \C^{n-1}$ with $y=(y_2,\ldots, y_n)$ and use multi-indices notations. Namely, $y^{J}=y_2^{j_2}\cdots y_n^{j_n}$, if $J=(j_2,\ldots,j_n)$; for a multi-index $I=(i_1,\ldots, i_n)$ we let $|I|=\sum_{j=1}^n i_j$.

We let
\begin{equation}\label{Q2}
Q_2(x,y)=\left( \sum_{I=(i_1,J)\in \N^n, |I|=2} q^1_{i_1,J}x^{i_1}y^{J},\ldots, \sum_{I=(i_1,J)\in \N^n, |I|=2} q^n_{i_1,J}x^{i_1}y^{J} \right),
\end{equation}
for some $q^k_{i_1,J}\in \C$.

Now we characterize boundary jets of infinitesimal generators:

\begin{proposition}\label{basic}
Let $G:\B^n\to \C^n$ be an infinitesimal generator of class $C^3$ at $e_1$. Assume that $e_1$ is a  BRNP with dilation $\beta=0$. Let \eqref{expansion} be the expansion of $G$ at $e_1$, with $T$ given by \eqref{T} and $Q_2$ given by \eqref{Q2}. Then $\Re q^1_{2,0}\geq 0$ and $\Re s_{kk}\leq -|q^1_{0,e_{2k}}|$ for all $k=2,\ldots, n$.

Moreover,  $\Re g_v''(1)=0$ for all $v\in \mathcal L_{e_1}$ if and only if
\begin{equation}\label{cond}
\Re q^1_{2,0}=\Re s_{kk}=0, \quad k=2,\ldots, n,
\end{equation}
and, $g_v''(1)=0$ for all $v\in \mathcal L_{e_1}$ if and only if $q^1_{2,0}=s_{kk}=0, \quad k=2,\ldots, n$.

If \eqref{cond} holds, then
\begin{itemize}
\item $S$ is anti-Hermitian,
  \item $q^1_{0,J}=0$ for all $J\in \N^n$, $|J|=2$,
  \item $q^1_{1,e_k}=q^h_{0,e_k+e_h}$ for $h,k=2,\ldots, n$,
  \item $q^h_{0,e_k+e_l}=0$ for all $h,k,l=2,\ldots, n$ with $h\neq k, l$,
  \item $\Im q^k_{1,e_k}=\Im q^1_{2,0}=0$ and $\Re q^k_{1,e_k}\geq 0$ for $k=2,\ldots, n$,
  \item the matrix $\tilde{Q}:=(q^k_{1,e_h})_{h,k=2,\ldots, n}$ is Hermitian and positive semi-definite.
\end{itemize}
Moreover, there exists $\delta\leq 0$ such that
\begin{equation}\label{Q3a0}
Q^1_3(x,y)=\delta x^3+\sum_{j=2}^n (q^1_{1,e_j}+\overline{q}^j_{2,0}) x^2 y_j.
\end{equation}
Finally, if $\Re g_v''(1)=0$ for all $v\in \mathcal L_{e_1}$ then $g_v'''(1)=0$ for all $v\in \mathcal L_{e_1}$ if and only if $\Re q^k_{1,e_k}= 0$,  $q_{1,e_k}^h=0$ for $2\leq k<h\leq n$  and $\delta=0$.
\end{proposition}

\begin{proof}
Let $v\in \mathcal L_{e_1}$, and let $g_v:\D\to \C$ be the slice reduction of $G$ with respect to $v$. Let $G$ be given by \eqref{expansion}, with $Tv=(T^1v,T''v)\in \C\times \C^{n-1}$, $Q_2(v)=(Q_2^1(v), Q_2''(v))\in \C\times \C^{n-1}$ and  $Q_3(v)=(Q_3^1(v), Q_3''(v))\in \C\times \C^{n-1}$. By Theorem \ref{JWC},
\[
T^1v=\la Tv,e_1\ra=\la dG_{e_1}(v), e_1\ra =\al \beta =0.
\]
Thus, a direct computation from  \eqref{explicit} shows that
\[
g_v(\zeta)=a_v(\zeta-1)^2+b_v (\zeta-1)^3+o(|\zeta-1|^3),
\]
with
\begin{equation}\begin{split}\label{a_v}
    a_v&=Q_2^1(v)-\la T''v,v''\ra,\\
    b_v&=(1-\al^2)Q_2^1(v)+\al Q_3^1(v)-\al \la Q_2''(v),v''\ra.
\end{split}\end{equation}
Now, $g_v(\zeta)=a_v(\zeta-1)^2 +b_v(\zeta-1)^3+o(|\zeta-1|^3)$ is an infinitesimal generator in the unit disc and thus, by Berkson-Porta formula, it has to hold $\Re (a_v+b_v(\zeta-1)+o(|\zeta-1|))\geq 0$.
Therefore, in particular, $\Re a_v\geq 0$ for all $v\in \mathcal L_{e_1}$ (see also, \cite{Sh}).

By writing down explicitly the condition $\Re a_v\geq 0$, we find
that for all $v=(\al,v_2,\ldots, v_n)$ with $\al\in (0,1]$,
$\sum_{j=2}^n |v_j|^2=1-\al^2$,
\begin{equation}\label{jet2}
\begin{split}
&\sum_{2\leq j\leq k\leq n} \Re (q^1_{0,e_j+e_k}v_jv_k)-\sum_{j,k=2}^n \Re (s_{kj}v_j\overline{v}_k)\\&+\al \left[\sum_{j=2}^n \Re (q^1_{1,e_j} v_j)-\sum_{k=2}^n \Re (t_k\overline{v}_k)\right]+\al^2\Re q_{2,0}^1\geq 0.
\end{split}
\end{equation}
For $\al=1, v''=0$, we find $\Re q_{2,0}^1\geq 0$.

When $\al\to 0$, the previous inequality implies that the term of degree $0$ in $\al$ has to have real part $\geq 0$, namely
\begin{equation}\label{al0}
\sum_{2\leq j\leq k\leq n} \Re (q^1_{0,e_j+e_k}v_jv_k)-\sum_{j,k=2}^n \Re (s_{kj}v_j\overline{v}_k)\geq 0
\end{equation}
for $\|v''\|=1$. Now, fix $k\in\{2,\ldots, n\}$ and substitute $v''$ with $e^{i \theta_k}e_k$ for $\theta_k\in [0,2\pi]$. We obtain $\Re (q^1_{0,2e_k}e^{2i\theta_k})-\Re s_{kk}\geq 0$, which, for the arbitrariness of $\theta_k$,  implies $\Re s_{kk}\leq -|q^1_{0,2e_k}|$  for $k=2,\ldots, n$.

Now, $\Re g''_v(1)=0$ if and only if $\Re a_v=0$ for all $v\in \mathcal L_{e_1}$. Therefore, it is clear from the previous considerations that if $\Re g''_v(1)=0$  for all $v\in \mathcal L_{e_1}$, then necessarily $\Re q_{2,0}^1=0$. Moreover, the left hand side of \eqref{al0} is equal to $0$. Therefore, fixing  $k\in\{2,\ldots, n\}$ and substituting $v''$ with $e^{i \theta_k}e_k$ for $\theta_k\in [0,2\pi]$ we obtain $\Re (q^1_{0,2e_k}e^{2i\theta_k})-\Re s_{kk}= 0$. Integrating with respect to $\theta_k$ in $[0,2\pi]$ the harmonic term vanishes and we obtain $\Re s_{kk}= 0$ for $k=2,\ldots, n$.

Assume that $\Re s_{kk}=0$ for $k=2,\ldots, n$. Therefore the
non-harmonic part in \eqref{al0} is zero, and we claim that this
implies that \eqref{al0} is, in fact, identically $0$. Indeed, we
rewrite \eqref{al0} as
\[
 \sum_{2\leq j\leq k\leq n} \Re (q^1_{0,e_j+e_k}v_jv_k)-\sum_{2\leq j<k\leq n} \Re [(s_{kj}+\overline{s}_{jk})v_j\overline{v}_k]\geq 0.
\]
Taking $|v_k|=1, v_j=0$ for $j\neq k$, we find immediately $q^1_{0,2e_k}=0$, $k=2,\ldots, n$. Next, we take $v_2=\zeta$, $v_3=\pm \zeta$ with $|\zeta|=1/\sqrt{2}$ and $v_4=\ldots=v_n=0$ and we obtain
\[
\pm\left(\Re (q^1_{0,e_2+e_3}\zeta^2)-\Re (s_{32}+\overline{s}_{23})|\zeta|^2\right)\geq 0.
\]
Decoupling the harmonic and non-harmonic terms by integrating as before, we obtain
\[
\Re (s_{32}+\overline{s}_{23})=0, \quad q^1_{0,e_2+e_3}=0.
\]
Finally, taking $v_2=\zeta$, $v_3=e^{i\theta} \zeta$ with $|\zeta|=1/\sqrt{2}$, $\theta\in [0,2\pi]$ and $v_4=\ldots=v_n=0$ we obtain
\[
-|\zeta|^2\Re [(s_{32}+\overline{s}_{23})e^{-i\theta}]\geq 0
\]
which implies $s_{32}+\overline{s}_{23}=0$. A similar argument works for the other indices. This proves that $S$ is anti-Hermitian and $q^1_{0,J}=0$ for all $|J|=2$. Moreover, this proves that the terms of degree $0$ in $\al$ in \eqref{jet2} are identically zero. Therefore, the  condition $\Re s_{kk}=0$ for all $k$ is sufficient for \eqref{jet2} in degree zero in $\al$ to be equal to zero for all $v$  when $\al\to 0$.

Now, since the terms of degree $0$ in $\al$ in \eqref{jet2} are vanishing identically, the terms of degree $1$  in $\al$ has to have non negative real part as $\al\to 0$, that is
\begin{equation}\label{al1}
\sum_{k=2}^n \Re [(q^1_{1,e_k}-\overline{t}_k)v_k]\geq 0,
\end{equation}
which clearly implies $q^1_{1,e_k}=\overline{t}_k$ for $k=2,\ldots, n$. Hence, if $\Re s_{kk}=\Re q^1_{2,0}=0$ then $\Re g''_v(1)=0$ for all $v\in \mathcal L_{e_1}$.

From the previous considerations it follows easily that $g_v''(1)=0$ for all $v\in \mathcal L_{e_1}$ if and only if  $q^1_{2,0}=s_{kk}=0$ for $k=2,\ldots, n$.

Now, assume \eqref{cond}. Then for all $v\in \mathcal L_{e_1}$ we have  $\Re a_v=0$  namely, $g_v(\zeta)=(\zeta-1)^2[i\tilde{a}_v+b_v(\zeta-1)+o(|\zeta-1|)]$ for all $v\in \mathcal L_{e_1}$, where $\tilde{a}_v\in \R$. Berkson-Porta's formula (see \cite{Sh}) implies then $b_v\in \R$ and $b_v\leq 0$.

Taking into account what we have already proved,  writing
$Q^1_3(v)=\sum_{|(i_1,J)|=3} p^1_{i_1,J}\al^{i_1}v^J$, where we
used the multi-indices notation $v^J=v_2^{j_2}\cdots v_n^{j_n}$,
from \eqref{a_v}, the condition $b_v\leq 0$ becomes
\begin{equation}\label{jet3}
\begin{split}
&(1-\al^2)\left(q^1_{2,0}\al^2+\al \sum_{k=2}^n q^1_{1,e_k}v_k\right) +\al \sum_{|(i_1,J)|=3} p^1_{i_1,J}\al^{i_1}v^J\\&-\al \sum_{k=2}^n\sum_{|(i_1,J)|=2}q^k_{i_1,J}\al^{i_1}v^J\overline{v}_k\leq 0.
\end{split}
\end{equation}
For $\al=1, v''=0$, we obtain $p^1_{3,0}\leq 0$. And, if $g_v'''(1)=0$ for all $v$, that is $b_v=0$, then $p^1_{3,0}=0$.

Now, as before, we start looking at terms of smallest degree in $\al$ when $\al\to 0$. Since there are no terms of degree $0$ in $\al$, the smallest degree is $1$,  and we get
\begin{equation}\label{deg1a}
\sum_{k=2}^n q^1_{1,e_k}v_k+\sum_{|I|=3}p_{0,I}^1v^I-\sum_{k=2}^n\sum_{|J|=2}q^k_{0,J}v^J\overline{v}_k\leq 0,
\end{equation}
for all $v\in \C^{n-1}$ with $\|v\|=1$. Replacing $v$ by $-v$ the left-hand side of \eqref{deg1a} changes sign. Therefore we deduce that
\begin{equation}\label{deg1}
\sum_{k=2}^n q^1_{1,e_k}v_k+\sum_{|I|=3}p_{0,I}^1v^I-\sum_{k=2}^n\sum_{|J|=2}q^k_{0,J}v^J\overline{v}_k= 0,
\end{equation}
for all $v\in \C^{n-1}$ with $\|v\|=1$.
Replacing $v$ by $e^{i\theta}v$ for $\theta\in [0,2\pi]$ in \eqref{deg1}, dividing the equation by $e^{i\theta}$ and integrating with respect to $\theta$ in $[0,2\pi]$ the harmonic terms vanish and we obtain
\begin{equation}\label{deg1m}
\sum_{k=2}^n q^1_{1,e_k}v_k-\sum_{k=2}^n\sum_{|J|=2}q^k_{0,J}v^J\overline{v}_k= 0.
\end{equation}
Equation \eqref{deg1} implies then $\sum_{|I|=3}p_{0,I}^1v^I=0$, which is possible only if $p_{0,I}^1=0$ for all $|I|=3$.
Taking $v_k=e_k$ for $k=2,\ldots, n$ in \eqref{deg1m} we obtain $q^1_{1,e_k}=q^k_{0,2e_k}$  for $k=2,\ldots, n$.

Now, let $2\leq k_1< k_2\leq n$ and let $v=\frac{1}{\sqrt{2}}(e^{i\theta_1}e_{k_1}+e^{i\theta_2}e_{k_2})$ for $\theta_1,\theta_2\in [0,2\pi]$. Expanding \eqref{deg1m} with such a choice of $v$, multiplying by $2\sqrt{2}$ and taking into account that $q^1_{1,e_k}=q^k_{0,2e_k}$  for $k=2,\ldots, n$, we obtain
\begin{equation*}
\begin{split}
0=&e^{i\theta_1}(2q^1_{1,e_{k_1}}-q^{k_1}_{0,2e_{k_1}}-q^{k_2}_{0,e_{k_1}+e_{k_2}})+
e^{i\theta_2}(2q^1_{1,e_{k_2}}-q^{k_2}_{0,2e_{k_2}}-q^{k_1}_{0,e_{k_1}+e_{k_2}})
\\&-e^{i(2\theta_2-\theta_1)}q^{k_1}_{0,2e_{k_2}}-e^{i(2\theta_1-\theta_2)}q^{k_2}_{0,2e_{k_1}}
\\ &=e^{i\theta_1}(q^1_{1,e_{k_1}}-q^{k_2}_{0,e_{k_1}+e_{k_2}})+
e^{i\theta_2}(q^1_{1,e_{k_2}}-q^{k_1}_{0,e_{k_1}+e_{k_2}})-e^{i(2\theta_2-\theta_1)}q^{k_1}_{0,2e_{k_2}}-e^{i(2\theta_1-\theta_2)}q^{k_2}_{0,2e_{k_1}},
\end{split}
\end{equation*}
from which we deduce that $q^1_{1,e_{k}}=q^{h}_{0,e_{k}+e_{h}}$ and $q^{k}_{0,2e_{h}}=0$ for $k\neq h\in \{2,\ldots, n\}$.

Finally, let $2\leq k_1<k_2<k_3\leq n$ and  consider $v=\frac{1}{\sqrt{3}}(e^{i\theta_1}e_{k_1}+e^{i\theta_2}e_{k_2}+e^{i\theta_3}e_{k_3}$ for $\theta_1,\theta_2,\theta_3\in [0,2\pi]$. Expanding \eqref{deg1m} with such a choice of $v$,  multiplying by $3\sqrt{3}$,  we obtain
\begin{equation*}
\begin{split}
0&=e^{i\theta_1}(3q^1_{1,e_{k_1}}-q^{k_1}_{0,2e_{k_1}}-q^{k_2}_{0,e_{k_1}+e_{k_2}}-q^{k_3}_{0,e_{k_1}+e_{k_3}})+
e^{i\theta_2}(3q^1_{1,e_{k_2}}-q^{k_2}_{0,2e_{k_2}}-q^{k_1}_{0,e_{k_1}+e_{k_2}}-q^{k_3}_{0,e_{k_2}+e_{k_3}})\\&+
e^{i\theta_3}(3q^1_{1,e_{k_3}}-q^{k_3}_{0,2e_{k_3}}-q^{k_1}_{0,e_{k_1}+e_{k_3}}-q^{k_2}_{0,e_{k_2}+e_{k_3}})-
e^{i(2\theta_2-\theta_1)}q^{k_1}_{0,2e_{k_2}}-e^{i(2\theta_3-\theta_1)}q^{k_1}_{0,2e_{k_3}}\\&-
e^{i(2\theta_1-\theta_2)}q^{k_2}_{0,2e_{k_1}}-e^{i(2\theta_3-\theta_2)}q^{k_2}_{0,2e_{k_3}}-
e^{i(2\theta_1-\theta_3)}q^{k_3}_{0,2e_{k_1}}-e^{i(2\theta_2-\theta_3)}q^{k_3}_{0,2e_{k_2}}\\
&-e^{i(\theta_1+\theta_3-\theta_2)}q^{k_2}_{0,e_{k_1}+e_{k_3}}-e^{i(\theta_2+\theta_3-\theta_1)}q^{k_1}_{0,e_{k_2}+e_{k_3}}-
e^{i(\theta_1+\theta_2-\theta_3)}q^{k_3}_{0,e_{k_1}+e_{k_2}}=0.
\end{split}
\end{equation*}
The first three lines of the previous equation do not give any new information, but the last one implies
that $q^{k}_{0,e_{h}+e_{l}}=0$ for $k,h,l=2,\ldots,n$ and $k\neq  h,l$.

Therefore, the term of degree $1$ in $\al$, for $\al\to 0$ in
\eqref{jet3} identically vanishes. So we look at terms of degree
$2$ in $\al$ as $\al\to 0$. We have
\begin{equation}\label{deg2}
q^1_{2,0}+\sum_{|I|=2}p^1_{1,I}v^I-\sum_{j,k=2}^n q^k_{1,e_j}v_j\overline{v}_k\leq 0,
\end{equation}
for all $v\in \C^{n-1}$ with $\|v\|=1$.
Replacing $v_j$ with $e^{i\theta_j}v_j$ for $\theta_j\in [0,2\pi]$ and integrating, we get rid of the harmonic terms and we find
\[
q^1_{2,0}-\sum_{k=2}^n q_{1,e_k}^k|v_k|^2\leq 0.
\]
Taking $v=e_k$  we obtain $q^1_{2,0}-q^k_{1,e_k}\leq 0$. Since $\Re q^1_{2,0}=0$, this implies $\Im q^k_{1,e_k}=\Im q^1_{2,0}$ and $\Re q^k_{1,e_k}\geq 0$ for $k=2,\ldots, n$.

Now, the harmonic part in \eqref{deg2} must be real, that is
\[
\sum_{2\leq j\leq l\leq n}p^1_{1,e_j+e_l} v_jv_l-\sum_{k=2}^n\sum_{j=2, j\neq k}^n q^k_{1,e_j}v_j\overline{v}_k\in \R.
\]
Taking $v=e_k$, this immediately implies $p^1_{1,2e_{k}}=0$. Taking $v=\frac{1}{\sqrt{2}}e^{i\theta}(e_{k_1}+e_{k_2})$ with $2\leq k_1<k_2\leq n$ and $\theta\in [0,2\pi]$, we obtain
\[
e^{2i\theta}p^1_{1,e_{k_1}+e_{k_2}}-(q^{k_1}_{1,e_{k_2}}+q^{k_2}_{1,e_{k_1}})\in \R,
\]
which implies $p^1_{1,e_j+e_k}=0$ for $j,k=2,\ldots, n$, $j\neq k$. Next, setting $v=\frac{1}{\sqrt{2}}(e_{k_1}+e^{i\theta}e_{k_2})$
with $\theta\in [0,2\pi]$, $2\leq k_1< k_2\leq n$, we obtain $(e^{-i\theta}q^{k_1}_{1,e_{k_2}}+e^{i\theta}q^{k_2}_{1,e_{k_1}})\in \R$, that is
\[
\Im [e^{-i\theta}(q_{1,e_{k_2}}^{k_1}-\overline{q}^{k_1}_{1,e_{k_2}})]=0,
\]
hence $q_{1,e_{k_2}}^{k_1}-\overline{q}^{k_2}_{1,e_{k_1}}=0$. This, together with \eqref{deg2} implies that the matrix $\tilde{Q}=(q_{1,e_{k}}^h)_{h,k=2,\ldots, n}$ is Hermitian and positive semi-definite.

From the previous considerations, we also note that  $b_v=0$ implies $\Re q^k_{1,e_k}= 0$ and $q_{1,e_k}^h=0$ for $2\leq k<h\leq n$,  while, this latter condition implies that the term of order $2$ in $\al$ as $\al\to 0$ in \eqref{jet3} vanishes identically.

Now, we are left to impose the condition that the remaining terms give a non-positive real number. Looking at terms of degree $3$ in $\al$ in \eqref{jet3}, we have
\[
0=\Im\left[\sum_{j=2}^n (-q^1_{1,e_j}v_j+p^1_{2,e_j}v_j-q^j_{2,0}\overline{v}_j)\right]=\sum_{j=2}^n\Im[(-q^1_{1,e_j}+p^1_{2,e_j}-\overline{q}^j_{2,0})v_j],
\]
for all $v\in \C^{n-1}$ with $\|v\|=1$, which clearly implies $-q^1_{1,e_j}+p^1_{2,e_j}-\overline{q}^j_{2,0}=0$ for all $j=2,\ldots, n$. In particular, the term of degree $3$ in $\al$ always vanishes identically.

Finally, we impose the condition that the terms of degree $4$ in $\al$ in \eqref{jet3}  are real. This means
\[
\Im \left[ -q^1_{2,0}+p^1_{3,0} \right]=0.
\]
Taking into account that we already proved that $p^1_{3,0}\leq 0$, this implies that $\Im q^1_{2,0}=0$. Note also that if the terms of degree $0,1,2,3$ in \eqref{jet3} vanish, which implies in particular that $q^1_{2,0}=0$, then the terms of order $4$ are vanishing if and only if $p^1_{3,0}=0$.

Summing up, if $\Re g_v''(1)=0$  all $v\in \mathcal L_{e_1}$ then $g_v'''(1)=0$ for all $v\in \mathcal L_{e_1}$ if and only if the terms of degree $2$ and $4$ in $\al$ in  \eqref{jet3} are identically vanishing (because those of degree $1$ and $3$ always do). This, in turn, is equivalent to  $\Re q^k_{1,e_k}= 0$,  $q_{1,e_k}^h=0$ for $2\leq k<h\leq n$  and $p^1_{3,0}=0$. And this ends the proof.
\end{proof}

\begin{proposition}\label{almost}
Let $G$ be an infinitesimal generator on $\B^n$, $C^3$ at $e_1$. Assume that $e_1$ is a BRNP with dilation $0$. Then $G$ generates a group of automorphisms if and only if the following conditions are satisfied:
\begin{enumerate}
  \item $\Re \la \frac{\de G}{\de z_k}(e_1),e_k\ra =0$, for $k=2,\ldots, n$,
  \item $\Re \la \frac{\de^2 G}{\de z_1\de z_k}(e_1),e_1\ra =0$,  for $k=1,\ldots, n$,
  \item $\la \frac{\de^2 G}{\de z_1\de z_k}(e_1),e_h\ra=0$ for $2\leq k<h\leq n$,
  \item $\Re \la \frac{\de^3 G}{\de z_1^3}(e_1),e_1\ra=0$.
\end{enumerate}
Moreover, if the previous conditions are satisfied, then $G\equiv 0$ if and only if $\la \frac{\de G}{\de z_k}(e_1),e_k\ra =0$ for $k=2,\ldots, n$ and $\la \frac{\de^2 G}{\de z_1^2}(e_1),e_1\ra =0$.
\end{proposition}

\begin{proof}
By Proposition \ref{basic}, the hypotheses are equivalent to the fact that for all $v\in \mathcal L_{e_1}$ the slice retraction $g_v$ has the property that $g_v'(1)=\Re g_v''(1)=g_v'''(1)=0$. By \cite[Corollary 4]{Sh} this is equivalent to the fact that $g_v$ is a generator of a group of automorphisms of $\D$ for all $v\in \mathcal L_{e_1}$ (and $g_v\equiv 0$ if and only if $g''_v(1)=0$). Then the statement follows from Proposition \ref{gruppo}.
\end{proof}

\begin{lemma}\label{trickP}
Let $G:\B^n\to \C^n$ be an infinitesimal generator, $C^3$ at $e_1$ and with BRNP at $e_1$ and dilation $\beta\in \R\setminus\{0\}$.
Let $H_\beta$ be given by \eqref{hhh}  and let $\tilde{G}:=G+H_\beta$. Then $\tilde{G}$ has a BRNP at $e_1$ with dilation $0$. Moreover, denoting by $\tilde{s}_{jk}$, $\tilde{q}^j_{i_1,J}$ the elements in the expansion of $\tilde{G}$ at $e_1$, we have $\tilde{s}_{kk}=s_{kk}-\frac{\beta}{2}$, $\tilde{q}^1_{2,0}=q^1_{2,0}-\frac{\beta}{2}$, $\tilde{q}^k_{1,e_k}=q^k_{1,e_k}-\frac{\beta}{2}$, $k=2,\ldots, n$ and $\tilde{q}^h_{1,e_k}=q^h_{1,e_k}$ for $2\leq k<h\leq n$.
\end{lemma}

\begin{proof}
By Corollary \ref{trick}, the vector field $G+H_\beta$ is an infinitesimal
generator with BRNP at $1$ and dilation $=0$. Now,
\begin{equation*}
\begin{split}
H_\beta(z)&=\left( -\beta (z_1-1),-\frac{\beta}{2}z_2,\ldots,-\frac{\beta}{2}z_n\right) \\&+ \left(-\frac{\beta}{2}(z_1-1)^2,-\frac{\beta}{2}(z_1-1)z_2,\ldots, -\frac{\beta}{2}(z_1-1)z_n \right).
\end{split}
\end{equation*}
From this the statements follow easily.
\end{proof}

\begin{proof}[Proof of Theorem \ref{rigidity}]
Let $\tilde{G}:=G+H_\beta$. Thanks to Lemma \ref{trickP}, the hypotheses on $G$ implies that $\tilde{G}$ satisfies the hypothesis of Proposition \ref{almost}, and the result follows.
\end{proof}

\section{On the quadratic expansion at BRNPs}\label{quadratic}

In \cite{Sh} it is shown that if $g:\D \to \C$ is an infinitesimal generator in $\D$ which is $C^3(1)$  with expansion $g(z)=z-1+a(z-1)^2+o(|z-1|^2)$ then the quadratic part $z\mapsto z-1+a(z-1)^2$ is always an infinitesimal generator in $\D$ which generates a semigroup of linear fractional self-maps of the unit disc.

In higher dimension the same result is false, and, even when the quadratic part is an infinitesimal generator, it might not generate  a semigroup of linear fractional maps. The underlying reason is that slice reductions at a BRNP of an infinitesimal generator do not preserve the degree of expansion at the boundary (cfr. \eqref{a_v}), so that the quadratic part of the infinitesimal generator in $\B^n$ might generate a cubic term on some slice reduction. We present the following examples.

\begin{example}
Let $F:\B^2\to \C^2$ be given by
\[
F(z_1,z_2)=-\left( z_1-1,\frac{5z_2}{4(2-z_2)}\right).
\]
We claim that $F$ is an infinitesimal generator. Indeed, for each
$z\in\de\B^2$ we have
\[
\langle -F(z),z \rangle
=|z_1|^2-\overline{z_1}+\frac{5|z_2|^2}{4(2-z_2)}=(1-\overline{z_1})+|z_2|^2\left(
\frac{5}{4(2-z_2)}-1 \right).
\]
Hence
\begin{equation}\begin{split}
\Re\langle -F(z),z \rangle &\ge 1-\sqrt{1-|z_2|^2} +|z_2|^2\Re\left(
\frac{5}{4(2-z_2)}-1 \right)\\
&\ge|z_2|^2\left(\frac1{1+\sqrt{1-|z_2|^2}} +\frac{5}{4(2+|z_2|)}-1
\right)\\&=r^2\,\frac{5-3\sqrt{1-r^2}-4r\sqrt{1-r^2}}
{4(2+r)(1+\sqrt{1-r^2})}\,,
\end{split}\end{equation}
where $r=|z_2|$.

A standard computation shows that the expression
${5-3\sqrt{1-r^2}-4r\sqrt{1-r^2}}$ is positive on the segment
$[0,1]$. Therefore, one concludes that $\Re\langle F(z),z
\rangle<0$ for all $z\in\de\B^2$. Taking into account that $F$ is holomorphic past the boundary of $\B^2$, one can apply Theorem \ref{one-gv}.(3) with $\beta=0$, and Proposition \ref{cbDW} to see that $F$ is an infinitesimal generator on $\B^2$ (see also \cite[Corollary 7.1]{RS}).

On the other hand, denote by $\tilde F$ the quadratic expansion of $F$ at $e_1$, namely,
\[
\tilde F(z)=-\left(z_1-1,\frac{5z_2}8\left(1+\frac{z_2}2\right)
\right).
\]
For this mapping
\[
\langle -\tilde F(z),z \rangle
=|z_1|^2-\overline{z_1}+\frac{5|z_2|^2}8\left(1+\frac{z_2}2\right).
\]
In particular, at the point $z_1=\frac1{\sqrt 2},\
z_2=-\frac1{\sqrt 2}$ we have
\[
\langle \tilde F(z),z
\rangle=-\frac{13}{16}+\frac{37}{32\sqrt{2}}\approx 0.005>0.
\]
So, $\tilde F$ is not a semigroup generator on the ball $\B^2$.
\end{example}

\begin{example}
Let $F:\B^2\to \C^2$ be given by
\[
F(z_1,z_2)=-\left( z_1-1,\frac{3z_2}{(2-z_2)}\right).
\]
We claim that $F$ is an infinitesimal generator on $\B^2$. Indeed, for each
$z\in\de\B$ we have
\[
\langle -F(z),z
\rangle=|z_1|^2-\overline{z_1}+\frac{3|z_2|^2}{2-z_2}=(1-\overline{z_1})+|z_2|^2\left(
\frac{3}{2-z_2}-1 \right).
\]
Since the inequality $\Re\frac{1}{2-z_2}>\frac13$ holds for
all $z_2,\ |z_2|<1$, we conclude that $\Re-\langle F(z),z
\rangle>0$ for all $z\in\de\B$. As in the previous example  taking into account that $F$ is holomorphic past the boundary of $\B^2$, one can apply Theorem \ref{one-gv}.(3) with $\beta=0$, and Proposition \ref{cbDW} to see that $F$ is an infinitesimal generator on $\B^2$ (see also \cite[Corollary 7.1]{RS}).

On the other hand, denote by $\tilde F$ the the quadratic expansion of $F$ at $e_1$, namely,
\[
\tilde F(z)=-\left(z_1-1,\frac{3z_2}2\left(1+\frac{z_2}2\right)
\right).
\]
One can easily see that $\tilde F$ generates a semigroup of holomorphic self-maps of $\B^2$ which does not
consist of linear fractional maps (cfr \cite{BCDl}).
\end{example}

\end{document}